\documentclass[11pt]{article}
\usepackage{amsmath,amsthm,amsfonts,amssymb,mathrsfs,epsfig,bm}
\usepackage[usenames]{color}
\usepackage[colorlinks=true]{hyperref}

\newtheorem{theorem}{Theorem}%[section]
\newtheorem{lemma}{Lemma}

\theoremstyle{remark}

\def\N{\mathbb{N}}
\def\Z{\mathbb{Z}}

\def\R{\mathbb{R}}
\def\C{\mathbb{C}}

\def\SS{\mathcal{S}}

\renewcommand{\phi}{\varphi}
\renewcommand{\epsilon}{\varepsilon}

\definecolor{mygray}{gray}{0.9}

\definecolor{deeppink}{RGB}{255,20,147}

\long\def\symbolfootnote[#1]#2{\begingroup
\def\thefootnote{\fnsymbol{footnote}}\footnote[#1]{#2}\endgroup}
\newcommand{\keywords}[1]{ \noindent {\footnotesize
             {\small \em Keywords and phrases.} {\sc #1} } }
\newcommand{\ams}[2]{  \noindent {\footnotesize
             {\small \em AMS {\rm 2000} subject classifications.
             {\rm Primary {\sc #1}; secondary {\sc #2}} } } }

\input diagxy

\begin{document}

\title{\bf A multilinear algebra proof of the Cauchy-Binet formula
and a multilinear version of Parseval's identity} 
\author{\sc Takis Konstantopoulos}
\date{\small \em 1 May 2013}
\maketitle

\begin{abstract}
We give a short proof of the Cauchy-Binet determinantal formula using 
multilinear algebra by first generalizing it to an identity {\em not}
involving determinants.
By extending the formula to abstract Hilbert spaces we obtain,
as a corollary, a generalization of the classical Parseval identity.

\vspace*{2mm}

\vspace*{2mm}
\ams{15A15,15A69}{15A24,46C05}
%15A15   	Determinants, permanents, other special matrix functions 
%15A69   	Multilinear algebra, tensor products
%5A75   	Exterior algebra, Grassmann algebras
%15A24   	Matrix equations and identities
%46-XX 			Functional analysis 
%46C05   	Hilbert and pre-Hilbert spaces: geometry and topology (including spaces with semidefinite inner product)

\end{abstract}

\section{Introduction and overview}
\label{intro}
The classical Cauchy-Binet formula states that
if $A, B$ are two matrices over $\R$ (or any field) of sizes
$n \times N$, $N \times n$, respectively, with $n \le N$, then
\begin{equation}
\label{cbclassic}
\det(AB) %= \sum_{\sigma} \det(A_\sigma B^\sigma)
= \sum_{\sigma} \det(A_\sigma) \det(B^\sigma)
\end{equation}
where the sum is taken over all $\sigma=(\sigma_1 < \sigma_2 < \cdots <
\sigma_n)$, with $\sigma_i \in \{1, \ldots, N\}$, and where $A_\sigma$
(respectively $B^\sigma$) is
the $n \times n$ submatrix of $A$ (respectively submatrix of $B$)
obtained by deleting all columns (respectively all rows)
except those with indices in $\sigma$.

There are many proofs of this formula, each telling its own story,
explaining the formula from a different point of view. The most
direct way of proving the formula is by writing down the 
determinant as a sum over permutations and performing algebraic
manipulations. This is the approach taken in many linear algebra books;
see, e.g., Marcus and Minc \cite[Theorem 6.1, p.\ 128]{MM}
and Gohberg {\em et al.} \cite[Theorem A.2.1, p.\ 651]{GLR}.
A probabilistic interpretation and proof of the formula (which starts 
by using the formula for a determinant) is also available \cite{EHR,ER}.
On the other hand, there are many combinatorial proofs. Suffice, perhaps,
to refer to the one chosen to be included in the ``Proofs from The Book''
\cite{PFTB} by Aigner and Ziegler. 
This is a nice proof (after all, it is a proof from The Book)
based on the beautiful Gessel-Vienot lemma 
which states that, in a finite
weighted acyclic directed graph,
the determinant of the path matrix between
two sets of vertices of cardinality $n$ each
equals a sum over all possible vertex-disjoint path systems;
see \cite[Chap.\ 29, p.\ 196]{PFTB} and \cite{AIG} for details.
Another very simple proof appears in the recent book by Terence Tao
\cite[p.\ 298]{TAO}) on random matrices. This proof is based on 
a relation between the characteristic polynomials of $AB$ and $BA$.

On the other hand, it is well-known that the Cauchy-Binet formula
is a generalization of the Pythagorean theorem. Indeed, let $A$ be 
a $n \times N$ real matrix, $n \le N$, and take $B=A^T$,
the transpose of $A$. Since $B^\sigma = (A^T)^\sigma = (A_\sigma)^T$,
the formula gives
\[
\det(AA^T) = \sum_\sigma \det(A_\sigma)^2,
\]
which can be interpreted geometrically as follows: The parallelotope in $\R^N$
generated by the $n$ row vectors of $A$ has $n$-dimensional
Lebesgue measure $\sqrt{\det(AA^T)}$.
Therefore the formula says that the square of the $n$-dimensional measure of 
an $n$-dimensional parallelotope, embedded in
a higher-dimensional Euclidean space, equals the sum of the
squares of the measures of its projections onto all possible $n$-dimensional
coordinate hyperplanes. If $n=1$ this reduces to the Pythagorean theorem.

The goal of this short article is to give a proof of the Cauchy-Binet
formula which is as simple as possible, from an algebraic-geometric 
viewpoint. If $n=1$, the Cauchy-Binet formula is a triviality: it states
that the inner product of two $N$-dimensional vectors equals
the sum of the products of their components:
\[
(a_1\ldots,a_N)\cdot (b_1, \ldots, b_N)^T = \sum_{\sigma=1}^N a_\sigma b_\sigma.
\]
There is no need to take determinants here, because both sides
involve $1\times 1$ matrices, i.e., real numbers.
What we show is that the general case, when $n \ge 1$, is the same,
but on bigger vector spaces.
In Section \ref{ingred} we give an account of the ingredients we need,
and, in Section \ref{centersec}, we state and prove the main formula (Theorem 
\ref{abstractcb}) without determinants and in a more general setup;
a corollary of it is the classical Cauchy-Binet formula.
Then, in Section \ref{HS}, we see that the formula can be extended to
a Hilbert space, giving a generalization of the classical 
Parseval identity. We conclude with a few bibliographic remarks.

\section{The main ingredients}
\label{ingred}
The main theorem, Theorem \ref{abstractcb} below, is requires
two ingredients.

(i) The first is the notion of the determinant of a linear transformation
$F: X \to X$ on a vector space $X$ of dimension $d$.
The dimension of the linear space $\bigwedge^m X$ of alternating $m$-linear
maps $\omega: X^m \to \R$
%, taking $[x_1, \ldots, x_m]$ to $\omega[x_1, \ldots, x_m]$, 
is $\binom{d}{m}$. 
For each $m$, 
the $m$-th level dual $F^* : \bigwedge^m X \to \bigwedge^m X$ of $F$ 
is defined by
\begin{equation}
\label{Fstar}
F^*\omega[x_1, \ldots, x_m] := \omega[Fx_1, \ldots, Fx_m].
\end{equation}
See, e.g., \cite{SPI}.
(Duals obey the standard composition rules: $(GF)^*=F^* G^*$.)
Since $\bigwedge^d X$ is $1$-dimensional, the $d$-th level dual
$F^*$ is multiplication by a constant.
This constant is, {\em by definition}, the determinant of $F$:
\begin{equation}
\label{detdef}
F^* \omega= (\det F) \cdot \omega, \quad \omega \in \textstyle \bigwedge^d X.
\end{equation}

(ii) The second ingredient is very simple too. Let $X, Y, Z$ be vector spaces,
and $F: X \to Y$, $G: Y \to Z$ linear maps. Suppose $Y$ is
the direct sum of $Y_1, \ldots, Y_K$.
Let $P_i: Y \to Y_i$, $1 \le i \le K$, be the projections corresponding
to this direct sum (so $\text{id}_V = P_1+\cdots+P_K$ is a partition
of the identity on $V$), and let $E_i : Y_i \to Y$ be the natural
embedding of $Y_i$ into $Y$. Then, clearly,
\begin{equation}
\label{split}
GF = \sum_{i=1}^K (GE_i) (P_i F).
\end{equation}
See Diagram 1.

\section{An abstract version of the Cauchy-Binet formula}
\label{centersec}
Let $U, V, W$ be finite-dimensional vector spaces of arbitrary
dimensions, and let 
$B : U \to V$, $A: V \to W$ be two linear maps.
Fix $n \in \N$ and consider the $n$-th level duals
$B^*: \bigwedge^n V \to \bigwedge^n U$, $A^*: \bigwedge^n W \to \bigwedge^n V$.
Let $N$ be the dimension of $V$ and let $f_1,\ldots,f_N$ be a basis for $V$.
See Diagram 2.
Denote by $\SS_n(N)$ the set of subsets of $\{1,\ldots,N\}$ of size $n$.
For each $\sigma \in \SS_n(N)$, let
$V_\sigma$ be the subspace of $V$ spanned by $\{f_i, i \in \sigma\}$
and consider the direct sum
\begin{equation}
\label{splitV}
V = V_\sigma \oplus V_{\overline \sigma},
\end{equation}
where $\overline \sigma :=\{1,\ldots,N\} \setminus \sigma$, letting
\[
P_\sigma: V \to V_\sigma
\]
be the projection of $V$ onto $V_\sigma$ along $V_{\overline \sigma}$,
and 
\[
E_\sigma: V_\sigma \to V
\]
the natural embedding of $V_\sigma$ into $V$.

%Let
%\[
%S_n(N) := \{\sigma=(\sigma_1, \ldots, \sigma_n):~
%\sigma_1 < \cdots < \sigma_n,~ 1\le \sigma_i \le N \text{ for
%all } i=1, \ldots, N\}.
%\]
%Given $\sigma \in S_n(N)$, let $V_\sigma$ be the subspace of $V$
%spanned by $f_{\sigma_1}, \ldots, f_{\sigma_n}$.
%Let $\overline \sigma$ be the strictly increasing
%sequence of length $N-n$ with elements in the set
%$\{1,\ldots, N\} \setminus \{\sigma_1, \ldots, \sigma_n\}$.
%We also let $V_\sigma$ be the span of $\{f_{\sigma_1}, \ldots, f_{\sigma_n}\}$;
%Similarly for $V_{\overline \sigma}$.
%We have 
%\begin{equation}
%\label{splitV}
%V = V_\sigma \oplus V_{\overline \sigma}.
%\end{equation}
%Let
%\[
%P_\sigma: V \to V_\sigma
%\]
%be the projection corresponding to this
%direct sum. Let
%\[
%E_\sigma: V_\sigma \to V
%\]
%be the natural embedding of $V_\sigma$
%into $V$.
\begin{theorem}
\label{abstractcb}
\begin{equation}
\label{cb}
(AB)^* = \sum_{\sigma \in \SS_n(N)} (P_\sigma B)^* (A E_\sigma)^*,
\end{equation}
\end{theorem}
\begin{proof}
The $\binom{N}{n}$--dimensional space $\bigwedge^n V$
is the direct sum of the $1$-dimensional spaces $\bigwedge^n V_\sigma$,
where $\sigma$ ranges in $\SS_n(N)$:
\begin{equation}
\label{splitwedge}
\textstyle \bigwedge^n V = \bigoplus_{\sigma \in \SS_n(N)} \bigwedge^n V_\sigma.
\end{equation}
%Therefore each $\omega \in \bigwedge^n V$ can be uniquely
%written as a sum of its components $\mathscr P_\sigma \omega$,
%$\sigma \in S_n(N)$. So $\mathscr P_\sigma: \bigwedge^n V
%\to \bigwedge^n V_\sigma$ are projection operators corresponding 
%to the above direct sum.
Let $\mathscr P_\sigma: \bigwedge^n V \to \bigwedge^n V_\sigma$
be projections corresponding to this direct sum,
and let  $\mathscr E_\sigma : \bigwedge^n V_\sigma
\to \bigwedge^n V$ be natural embedding.
Using \eqref{split} (with $A^*$, $B^*$ in place of $F$, $G$, respectively,
and $K=\binom{N}{n}$), we have
\[
B^* A^* 
 = \sum_{\sigma \in \SS_n(N)} (B^* \mathscr E_\sigma) (\mathscr P_\sigma A^*).
\]
%We therefore have, for each $\eta \in \bigwedge^n W$,
%$\omega=A^*\eta$,
%\[
%A^* \eta = \sum_{\sigma \in S_n(N)} \mathscr P_\sigma A^* \eta,
%\]
%and so
%\[
%B^* A^* \eta = \sum_{\sigma \in S_n(N)} B^* \mathscr P_\sigma A^* \eta
 %= \sum_{\sigma \in S_n(N)} B^* \mathscr E_\sigma \mathscr P_\sigma A^* \eta,
%\]
%where $\mathscr E_\sigma : \bigwedge^n V_\sigma
%\to \bigwedge^n V$ is the natural embedding.
%Since (see Lemma \ref{lll} below)
%\[
%\mathscr E_\sigma = P_\sigma^*, \quad 
%\mathscr P_\sigma = E_\sigma^*,
%\]
%we have
%\[
%(AB)^* = B^*A^* = \sum_{\sigma \in S_n(N)} B^* \mathscr E_\sigma \mathscr P_\sigma A^*
%= \sum_{\sigma \in S_n(N)} B^* P_\sigma^* E_\sigma^* A^*
%= \sum_{\sigma \in S_n(N)} (P_\sigma B)^* (A E_\sigma)^*,
%\]
%proving the theorem.
See Diagram 2.
Since (see Lemma \ref{lll} below) 
\[
\mathscr E_\sigma = P_\sigma^*, \quad 
\mathscr P_\sigma = E_\sigma^*,
\]
the theorem follows from the 
composition rules of the duals.
\end{proof}

\begin{lemma}
\label{lll}
\begin{equation*}
%\label{duals}
\mathscr P_\sigma = E_\sigma^*,
\quad
\mathscr E_\sigma = P_\sigma^*.
\end{equation*}
\end{lemma}
\proof
We identify $\SS_n(N)$ with the set of strictly
increasing sequences of length $n$ with values in $\{1,\ldots,N\}$.
Thus, if $\sigma$ is a subset of $\{1,\ldots,N\}$ we 
let $(\sigma_1, \ldots, \sigma_n)$ be a listing of
its elements in increasing order.
To prove the first equality it suffices to show that
\[
\mathscr P_\sigma \omega[v_1, \ldots, v_n]
= \omega[E_\sigma v_1, \ldots, E_\sigma v_n],
\]
for all $\omega \in \Lambda_n(V)$ and all $v_1, \ldots, v_n \in V_\sigma$.
But then $E_\sigma v_i = v_i$ and, since $V_\sigma$ is spanned
by $f_{\sigma_1}, \ldots, f_{\sigma_n}$, it suffices to
show that
\[
\mathscr P_\sigma \omega[f_{\sigma_{\pi(1)}}, \ldots, f_{\sigma_{\pi(n)}}]
= \omega[f_{\sigma_{\pi(1)}}, \ldots, f_{\sigma_{\pi(n)}}],
\]
where $\pi$ is a permutation of $\{1,\ldots, n\}$.
Since $\omega=\sum_{\tau \in \SS_n(N)} \mathscr P_\tau \omega$
[this is the partition of the identity on $\bigwedge^n V$ 
corresponding to \eqref{splitwedge}]
we may replace $\omega$ by $\mathscr P_\tau \omega$ in the last
display:
\[
\mathscr P_\sigma \mathscr P_\tau
\omega[f_{\sigma_{\pi(1)}}, \ldots, f_{\sigma_{\pi(n)}}]
= \mathscr P_\tau
\omega[f_{\sigma_{\pi(1)}}, \ldots, f_{\sigma_{\pi(n)}}].
\]
But then, if $\tau=\sigma$ the two sides are obviously equal,
and if $\tau \neq \sigma$ the left-hand side equals zero
and $\mathscr P_\tau \omega[f_{\sigma_{\pi(1)}}, \ldots, f_{\sigma_{\pi(n)}}]
=0$.

To prove the second equality it suffices to show that
\[
\mathscr E_\sigma \omega[v_1, \ldots, v_n]
= \omega[P_\sigma v_1, \ldots, P_\sigma v_n],
\]
for all $\omega \in \Lambda_n(V_\sigma)$
and all $v_1, \ldots, v_n \in V$.
But then $\mathscr E_\sigma \omega=\omega$. Since
$v_i=P_\sigma v_i + P_{\overline \sigma} v_i$
[corresponding to \eqref{splitV}], 
we have
\[
\mathscr E_\sigma \omega[v_1, \ldots, v_n]
= \omega [ P_\sigma v_1 + P_{\overline \sigma} v_1,\ldots,
P_\sigma v_n + P_{\overline \sigma} v_n].
\]
Using the multilinearity of $\omega$ we split the latter
into $2^n$ terms, all of which are zero except the one involving
only $P_\sigma v_i$ as arguments.
\qed

Consider now the case where $W=U$. Moreover, take the number $n$ in
Theorem \ref{abstractcb} to be equal to their common dimension.
Assume $n \le N=\dim V$ to avoid trivialities.
Then the linear maps $(AB)^*$, $(P_\sigma B)^*$, and $(A E_\sigma)^*$, appearing
in formula \eqref{cb}, are
maps between $1$-dimensional spaces. 
Since the spaces $V_\sigma$ and $U$ have common dimension $n$, we can 
identify them by means of a linear bijection
\[
\phi_\sigma: V_\sigma \to U.
\]
Then
\[
(P_\sigma B)^* (A E_\sigma)^* =
(P_\sigma B)^* \phi_\sigma^* (\phi_\sigma^{-1})^* (A E_\sigma)^*
= (\phi_\sigma P_\sigma B)^* (A E_\sigma \phi_\sigma^{-1})^*,
\]
and so
\begin{equation}
\label{cbU}
(AB)^* = \sum_{\sigma \in \SS_n(N)} (\phi_\sigma P_\sigma B)^* (A E_\sigma \phi_\sigma^{-1})^*.
\end{equation}
Since all three linear maps $AB$, $\phi_\sigma P_\sigma B$, 
$A E_\sigma \phi_\sigma^{-1}$ are linear maps on the same $1$-dimensional
vector space $U$, it follows, from the definition of the determinant,
that
\begin{equation}
\label{cbUdet}
\det(AB) = \sum_{\sigma \in \SS_n(N)} \det(\phi_\sigma P_\sigma B)
\det(A E_\sigma \phi_\sigma^{-1}).
\end{equation}
(The role of $\phi_\sigma$ is to force all maps
be on the same space, so we can talk about determinants.)
In the case where $U=\R^n$, $V=\R^N$, this proves the classical Cauchy-Binet
formula \eqref{cbclassic}.
If $N=n$, then we have shown that the determinant of the product
is the product of the determinants.

Therefore \eqref{cbclassic} follows from \eqref{cbUdet}. The latter is
a restatement of \eqref{cbU}. But \eqref{cbU} is a special case of 
\eqref{cb} because in \eqref{cb} we allow $U, V, W$ to be
different with dimensions that may be distinct from $n$.

\section{Multilinear Parseval's identity}
\label{HS}
We are now going to replace the middle space $V$ of the previous setup
by a separable Hilbert space $H$ over the complex numbers $\C$, having
inner product $\langle x, y \rangle$.
Let $f_1, f_2, \ldots$ be an orthonormal basis for $H$. 
Let $\bigwedge^n H$ be the collection of all {\em continuous}
alternating multilinear functionals $\omega: H^n \to \C$.
In particular, $\bigwedge^1 H=H^*$ is the Hilbert space dual of $H$.
By the Riesz-Fischer theorem, $f_1, f_2, \ldots$ forms a basis for $\bigwedge^1 H$
in the sense that every $\omega \in \bigwedge^1 H$
can be uniquely written as $\omega[x] = \sum_{\sigma=1}^\infty
a_\sigma \langle f_\sigma, x \rangle$, for 
$a_\sigma \in \C$ such that $\sum_\sigma |a_\sigma|^2 < \infty$.
More generally, $\bigwedge^n H$ is a separable Hilbert space 
with orthonormal (with respect to a suitably defined inner product) basis
\[
f_{\sigma_1} \wedge \cdots \wedge f_{\sigma_n},
\quad \sigma=(\sigma_1, \ldots, \sigma_n) \in \SS_n(\N),
\]
where $\SS_n(\N)$ is the collection of all $n$-tuples $(\sigma_1, \ldots,
\sigma_n)$ of positive integers such that $\sigma_1 < \cdots < \sigma_n$.
Recall that the wedge product satisfies, by definition,
\[
(f_1 \wedge f_2) [x,y] = f_1[x] f_2[y] - f_1[y] f_2[x],
\]
and, more generally, $f_{\sigma_1} \wedge \cdots \wedge f_{\sigma_n}$
is obtained by antisymmetrization of the tensor product of 
$f_{\sigma_1}, \ldots, f_{\sigma_n}$.
Incidentally, the direct sum of $\bigoplus_{n=0}^\infty \bigwedge^n H$ 
(where
$\bigwedge^0 H := \C$) is the so-called alternating Fock (or fermionic)
space \cite{RS}. 
Wedge products can be defined, by linearity, between any finite
number of elements of this space.

If $H_1, H_2$ are two Hilbert spaces and 
$F: H_1 \to H_2$ is a continuous linear function
then $F^* : \bigwedge^n H_2 \to \bigwedge^n H_1$
is defined as before--see \eqref{Fstar}--and is, moreover, continuous.

\begin{theorem}
\label{HT}
Let $H$ be a separable Hilbert space over $\C$ 
with orthonormal basis $f_1, f_2, \ldots$,
and let $n$ be a positive integer.
For each $\sigma \in \SS_n(\N)$, let $H_\sigma$ be the subspace spanned by
$f_{\sigma_1}, \ldots, f_{\sigma_n}$. Let $E_\sigma : H_\sigma \to H$ be the
natural embedding of $H_\sigma$ into $H$
and $P_\sigma : H \to H_\sigma$ the orthogonal projection
of $H$ onto $H_\sigma$.
If $U$, $W$ are finite-dimensional vector spaces over $\C$ and $B: U \to H$,
$A: H \to W$ continuous linear maps,
then
\[
(AB)^* = \sum_{\sigma \in \SS_n(\N)} (P_\sigma B)^* (A E_\sigma)^*.
\]
If $W=U$ with common dimension $n$, and if $\phi_\sigma : H_\sigma \to U$ is 
any linear bijection, then
\[
(AB)^* = 
\sum_{\sigma \in \SS_n(\N)} (\phi_\sigma P_\sigma B)^* (A E_\sigma\phi_\sigma^{-1})^*.
\]
In particular,
\[
\det(AB) = \sum_{\sigma \in \SS_n(\N)} 
\det (\phi_\sigma P_\sigma B)
\det (A E_\sigma\phi_\sigma^{-1}).
\]
\end{theorem}
The proof of this theorem is exactly as in the finite-dimensional
case. Infinite sums have to be understood in the Hilbert space sense.

Consider now $H=L^2[0,1]$ with inner product $\langle x, y \rangle
= \int_0^1 x(t) \overline {y(t)} dt$ and the standard orthonormal basis
$e_k(t) = \exp( i 2\pi k t)$, $k \in \Z$, and let $U=W=\C^n$,
for a given positive integer $n$.
A continuous linear map $A: L^2[0,1] \to \C^n$ is
necessarily (Riesz representation theorem) of the form
\[
Ax = \big(\langle x, a_1 \rangle, \ldots, \langle x, a_n \rangle\big)
=  \bigg( \int_0^1 \overline{a_1(t)} x(t) dt, \ldots ,
\int_0^1 \overline{a_n(t)} x(t) dt \bigg),
\quad x \in L^2[0,1],
\]
where $a_1, \ldots, a_n \in L^2[0,1]$.
A linear map $B: \C^n \to L^2[0,1]$ is of the form
\[
(B u) (t) = u_1 b_1(t) + \cdots + u_n b_n(t),
\quad u \in \C^n,
\]
where $b_1, \ldots, b_n \in L^2[0,1]$.
Hence the $jk$-entry of the matrix of $AB : \C^n \to \C^n$,
with respect to the standard basis on $\C^n$, is given by
\[
(AB)_{jk} = \int_0^1 \overline{a_j(t)} b_k(t) dt.
\]
Consider now $\sigma \in \SS_n(\Z)$, i.e., $\sigma=(\sigma_1, \ldots, \sigma_n)
\in \Z^n$ with $\sigma_1 < \cdots < \sigma_n$.
(There is no difficulty in replacing $\N$ in the above theorem by $\Z$.)
Then $H_\sigma$ is the subspace of $L^2[0,1]$ spanned by $e_{\sigma_1},
\ldots, e_{\sigma_n}$.
So the orthogonal projection $P_\sigma: H \to H_\sigma$ is given by
\[
P_\sigma x = \widehat x(\sigma_1) e_{\sigma_1} + \cdots +
\widehat x(\sigma_n) e_{\sigma_n},
\]
where
\[
\widehat x(k) := \int_0^1 x(t) \exp(-i2\pi k t) dt, \quad k \in \Z,
\]
are the Fourier coefficients of $x$.
Letting $\phi_\sigma : H_\sigma \to \C^n$ be the linear bijection
that takes $e_{\sigma_r}$ into the $r$-th standard basis vector
of $\C^n$, for $r=1, \ldots, n$,
we see that the $jk$-entry of the matrix of $\phi_\sigma P_\sigma B$
is
\[
(\phi_\sigma P_\sigma B)_{jk} = \widehat b_k(\sigma_j).
\]
Arguing analogously, the $jk$-entry of the matrix of $A E_\sigma
\phi_\sigma^{-1}$ is
\[
(A E_\sigma \phi_\sigma^{-1})_{jk} = \overline{\widehat a_j(\sigma_k)}.
\]
Hence the last formula of Theorem \ref{HT} gives
\begin{align*}
\det_{1 \le j,k \le n} \int_0^1 \overline{a_j(t)} b_k(t) dt
&= \sum_{\sigma \in \SS_n(\Z)}
\det_{1 \le j,k \le n} \big[\overline{\widehat a_j(\sigma_k)}\big]
\det_{1 \le j,k \le n} \big[\widehat b_j(\sigma_k)\big]
\\
&= 
\frac{1}{n!} \sum_{\sigma_1 \in \Z} \cdots \sum_{\sigma_n \in \Z}
\det_{1 \le j,k \le n} \big[\overline{\widehat a_j(\sigma_k)}\big]
\det_{1 \le j,k \le n} \big[\widehat b_j(\sigma_k)\big],
\end{align*}
where the second equality follows from the fact that applying
the permutation of
$(\sigma_1, \ldots, \sigma_n)$ to both matrices
will change the sign of both determinants simultaneously
and the fact that repeated indices result into zero determinants.
For $n=1$, this is the standard Parseval identity.

Of course, there is nothing special with the Lebesgue measure. We
can obtain formulas for any other $L^2$ space
or other separable Hilbert spaces.

\section{Remarks}
My motivation for this article was due to my 
%I was motivated to understand the Cauchy-Binet formula
%because of my 
desire to understand some elements of random matrix
theory \cite{AGZ2010}
and determinantal point processes \cite{HKPV}. In particular,
the derivation of the ubiquitous Tracy-Widom 
probability distribution \cite{AGZ2010}
involves several applications of Cauchy-Binet type formulas.
When I looked at it first, a standard computational proof was not
too satisfactory. I discovered that there are many proofs,
which can be roughly classified into combinatorial and 
algebraic ones. The version presented in this short article
was inspired by the simple observation that the Cauchy-Binet formula
is a version of Pythagorean theorem: it is a version
of the Pythagorean theorem on $\bigwedge^n \R^N$, with $n \le N$
(which is of course isomorphic to $\R^{\binom{N}{n}}$).
%Indeed, when we start thinking of alternating multilinear forms
%as vectors, then the left-hand side of the Cauchy-Binet theorem
%is a natural inner product.

Several years ago, Zeilberger \cite{ZEIL} ``complained'' that,
to most contemporary mathematicians, matrices and linear transformations
are practically  interchangeable notions and that the mainstream
`Bourbakian' establishment, with its profound disdain for the concrete,
goes as far as to frown at the mere mention of the word `matrix'. 
He then explains how ``to [him], as well as to other `dissidents' called
`combinatorialists', a matrix has nothing whatsoever to do with that
intimidating abstract concept called `a linear transformation between linear
vector spaces' '' and, by 
thinking of matrices as putting weights on a graph,
he develops a combinatorial way of interpreting and proving fundamental results
such as the Cayley-Hamilton theorem.
The Cauchy-Binet formula has found a nice proof, in the Zeilberger sense,
as a corollary of the Gessel-Vienot lemma.
We also mention Zeng's proof \cite{ZENG} which also uses Zeilberger's methods.

In a sense then, what we have done here is in 
exactly the opposite of Zeilberger's spirit,
because the proof presented uses nothing else but the concept
of a linear map between vector spaces (and lots of
definitions). 
Each point of view has its own merits in that, for instance, it leads to
different kind of extensions.
(Extensions to infinite matrices are not easy when the combinatorial point
of view is adopted.)

We finally remark that there are
generalizations of the Cauchy-Binet formula for the case
where the matrices contain elements of a noncommutative ring
\cite{CSS}. We do not know how to extend the ideas above to this case.

\subsubsection*{Acknowledgments}
I thank Svante Janson \cite{SJ} for pointing out reference \cite{CSS} to me
and for his comments on this article, and Richard Ehrenborg \cite{EHR}
for kindly making his notes available to me.

\vspace*{1cm}
\hfill
\noindent
\begin{minipage}[t]{6cm}
\small \sc
Takis Konstantopoulos
\\
Department of Mathematics
\\
Uppsala University
\\
751 06 Uppsala
\\
Sweden
\\
{\tt takis@math.uu.se
\\
www.math.uu.se/$\sim$takis}
\end{minipage}

\vspace*{1cm}

\hrulefill

\begin{center}
\begin{tabular}{lll}
\hspace*{3cm}
&
\begin{minipage}{0.5\textwidth}
\begin{gather*}
\bfig
\morphism[X`Y;F]
\morphism(500,0)[~`Z;G]
\morphism(450,0)|l|<0,-500>[~`~;P_i]
\morphism(550,0)|r|/<-/<0,-500>[~`~;E_i]
\morphism(500,0)//<0,-570>[~`Y_i;]
\efig
\end{gather*}
\end{minipage}
&
\begin{minipage}{4cm}
\begin{gather*}
Y = \displaystyle \bigoplus_{i=1}^K Y_i
\\[2mm] 
GF = \sum_{i=1}^K (GE_i)(P_iF)
\end{gather*}
\end{minipage}
\end{tabular}
\\[2mm]
Diagram 1
\end{center}

\begin{center}
\begin{tabular}{lll}
\begin{minipage}{0.3\textwidth}
\begin{gather*}
\bfig
\morphism/<-/[W`V;A]
\morphism(500,0)/<-/[~`U;B]
\morphism(450,0)|l|/<-/<0,-500>[~`~;E_\sigma]
\morphism(550,0)|r|<0,-500>[~`~;P_\sigma]
\morphism(500,0)//<0,-570>[~`V_\sigma;]
\efig
\end{gather*}
\end{minipage}
&
\begin{minipage}{0.4\textwidth}
\begin{gather*}
\bfig
\morphism(-600,0)<600,0>[\bigwedge^n W`\bigwedge^n V;A^*]
\morphism(80,0)<600,0>[~`\bigwedge^n U;B^*]
\morphism(-80,-30)|l|<0,-500>[~`~;\mathscr P_\sigma]
\morphism(50,-30)|r|/<-/<0,-500>[~`~;\mathscr E_\sigma]
\morphism(0,0)//<0,-600>[~`\bigwedge^n V_\sigma;]
\efig
\end{gather*}
\end{minipage}
&
\begin{minipage}{0.33\textwidth}
\begin{gather*}
\textstyle \bigwedge^n V =\displaystyle \bigoplus_{\sigma \in \SS_n(N)} \textstyle \bigwedge^n V_\sigma
\\[2mm]
B^* A^* 
 = \sum_{\sigma \in \SS_n(N)} (B^* \mathscr E_\sigma) (\mathscr P_\sigma A^*)
\end{gather*}
\end{minipage}
\end{tabular}
\\[2mm]
Diagram 2
\end{center}

\end{document}